\documentclass[runningheads]{llncs}

\usepackage{xcolor}
\usepackage{graphicx}
\usepackage{amsfonts}
\usepackage{fancyhdr}
\usepackage{comment}
\usepackage{amsfonts}
\usepackage{fancyhdr}
\usepackage{comment}

\usepackage{times}
\usepackage{amsmath}
\usepackage{changepage}
\usepackage{amssymb}
\usepackage{graphicx}%
\usepackage[T1]{fontenc}

\usepackage{times}
\usepackage{amsmath}
\usepackage{changepage}
\usepackage{amssymb}
\usepackage{graphicx}%
\usepackage[T1]{fontenc}
\newcommand{\Q}{\mathbb{Q}}
\newcommand{\R}{\mathbb{R}}

\newcommand{\Z}{\mathbb{Z}}
\newcommand{\N}{\mathbb{N}}

\newcommand{\A}{\mathcal{A}}
\newcommand{\B}{\mathcal{B}}

\newcommand{\oor}{\vee}
\newcommand{\0}{\emptyset}
\newcommand{\aaa}{\boldsymbol{a}}
\newcommand{\bbb}{\boldsymbol{b}}
\newcommand{\ccc}{\boldsymbol{c}}

\newcommand{\lom}{^{< \omega}} 
\newcommand{\E}{\exists}
\newcommand{\orb}{\text{orb}}

\newcommand{\iv}{^{-1}}
\newcommand{\LL}{\mathcal{L}}
\newcommand{\dom}{\text{dom}}

\newcommand{\Fin}{\text{Fin}}

\newcommand{\Th}{\text{Th}}
\newcommand{\qr}{\text{qr}}
\newcommand{\ord}{\text{ord}}

\begin{document}
\title{Elementarity of Subgroups and Complexity of Theories for Profinite Groups}
%
%
\author{Jason Block\thanks{This paper is an expansion of the extended abstract \cite{me}.}}

\authorrunning{J. Block}
%
\institute{CUNY Graduate Center, New York NY 10016, USA\\
\email{jblock@gradcenter.cuny.edu} }
\maketitle              
\begin{abstract}
Although $S_\infty$ (the group of all permutations of $\mathbb{N}$) is size continuum, both it and its closed subgroups can be presented as the set of paths through a countable tree. The subgroups of $S_\infty$ that can be presented this way with  finite branching trees are exactly the profinite ones. We use these tree presentations to find upper bounds on the complexity of the existential theories of profinite subgroups of $S_\infty$, as well as to prove sharpness for these bounds. These complexity results enable us to distinguish a simple subclass of profinite groups, those with \emph{orbit independence}, for which we find an upper bound on the complexity of the entire first order theory. Additionally, given a profinite subgroup $G$ of $S_\infty$ and a Turing ideal $I$ we define $G_I$ to be the set of elements in $G$ whose Turing degree lies in $I$. We examine to what extent and under what conditions $G_I$ will be an elementary subgroup of $G$. In particular, we construct a profinite group whose subgroup of computable elements is not elementary even for existential formulas.

\end{abstract}

\section{Introduction}
Traditional computable structure theory deals only with countable structures. As a result, it cannot be used to study most subgroups of $S_\infty$ (the group of all permutations of $\N$). However as described in Section \ref{def section}, a large class of these subgroups (specifically the \textit{closed} subgroups) can be presented as the set of paths through a countable tree, similarly to how the real numbers are presented in computable analysis \cite{ca}.   We will focus our attention on the subgroups of $S_\infty$ that can be presented as the paths through a finite branching tree (the \textit{compact} subgroups), which are exactly the profinite ones.  In \cite{Russell} Miller uses such a presentation to study the absolute Galois group of $\Q$, which is indeed a profinite group that can be viewed as a subgroup of $S_\infty$ after fixing an enumeration of the algebraic closure of $\Q$. A well known result from computable analysis is that the field of computable real numbers is a real closed field and thus an elementary substructure of the field of real numbers. In contrast, Theorem \ref{not always elementary} gives a profinite group whose subgroup of computable elements is not an elementary subgroup.

Much interest in profinite groups stems from their connection to Galois theory. All Galois groups are profinite, and every profinite group is the Galois group of some field extension \cite{Waterhouse}. However, the purpose of this paper is to examine profinite groups simply as groups in their own right. It can be difficult to get an effective handle on uncountable groups, but when such a group acts on $\N$ by permutations we are given the opportunity to do so. Thus, we restrict our attention to the profinite subgroups of $S_\infty$. 

Effective notions for profinite groups within the context of Galois theory were examined by Metakides and Nerode in \cite{m and n} and further by La Roche in \cite{la roche}. Following this work, effective notions for profinite groups in general were studied by Smith in \cite{smith}. The authors of \cite{la roche} and \cite{smith} define a profinite group to be \emph{recursively profinite} if it is isomorphic to the inverse limit of a uniformly computable sequence of finite groups and surjective homomorphisms. As we will see in Proposition \ref{compatible defs}, a profinite group $P$ is recursively profinite if and only if it is isomorphic to a subgroup $G$ of $S_\infty$ such that $T_G$ (the tree that represents $G$) is computable.  

There have been many recent developments in the field of computable analysis and topology,  much of which has been applied to profinite groups. In \cite{downey melinkov}, Downey and Melnikov show that a profinite group is recursively profinite if and only if it can be presented as a computably compact Polish space in which the group operation is computable.  This is further equivalent to having computable Haar measure, as shown by Pauly, Seon and Ziegler in \cite{pauly}. As for abelian profinite groups, Melnikov shows in \cite{melnikov} that such a group is recursively profinite if and only if its Pontryagin dual has a computable presentation. In \cite{koh} Koh, Melnikov and Ng compare several definitions of the more general notion of computable topological group and profinite groups are used to illustrate some of the differences. In \cite{g and m and} Greenberg, Melnikov, Nies and Turetsky examine what they call ``effectively closed'' subgroups of $S_\infty$. This notion is not equivalent to being a closed subgroup $G$ of $S_\infty$ with $T_G$ computable, as the authors construct an effectively closed profinite subgroup of $S_\infty$ that is not recursively profinite. In particular, no tree presentation $T_G$ of the group can have computable branching.

In Section \ref{main results} we use tree presentations to determine bounds on the complexity of the existential theory of  profinite  subgroups of $ S_\infty$. Note that since all elements of $S_\infty$ are functions from $\N\to \N$, an existential sentence about a subgroup is a $\Sigma^1_1$ statement. However, we will see that the existential theory of any profinite subgroup  $G$ of $S_\infty$ is $\Sigma^0_2$ relative to the degree of $T_G$. Additionally, if $G$ has 
\textit{orbit independence}, then the existential theory is $\Sigma^0_1$ relative to the degree of $T_G$. We will also show that these bounds are sharp. Specifically, there exists a profinite  $G$ with orbit independence such that $T_G$ is computable and the existential theory of $G$ is $\Sigma^0_1$ complete, and such a $G$ without orbit independence such that the existential theory is $\Sigma^0_2$ complete. Last, we show that the (entire) first order theory of a profinite  $G$ with orbit independence is $\Delta^0_2$ relative to the degree of $T_G$.

In Section \ref{Elementarity} we consider countable subgroups of profinite subgroups $G$ of $S_\infty$. In particular, given a Turing ideal $I$ we define $G_I$ to be the subgroup of $G$  whose elements (viewed as functions from $\N\to \N$) have Turing degree in $I$. We examine to what degree such a $G_I$ will be an {elementary} subgroup of $G$. As shown by Theorem \ref{not always elementary}, it possible to have that $G_I$ will not be elementary even for purely existential formulas. However, if either $G$ has orbit independence or $I$ is a {Scott ideal}, then $G_I$ will be elementary at least for existential formulas. Additionally if we have both that $G$ has orbit independence and that $I$ is a Scott ideal, then $G_I$ is an elementary subgroup of $G$. 

\section{Tree Presentations}
\label{def section}

\begin{definition}
Let  $G$ be a subgroup of $ S_\infty$. We define the tree $T_G$ to be the subtree of $\N\lom$ containing all initial segments of elements of $G$. That is,$$T_G:=\left\{\tau\in \N\lom: (\E g\in G,n\in\N)\left[\tau=g(0)g\iv(0)g(1)g\iv(1)\cdots g(n)g\iv(n)\right]\right\} $$ where $m\in \N$ is mapped to $g(m)$ under $g$. We define the ordering of $T_G$  via initial segments and write $\tau \sqsubset \sigma$ if $\tau$ is an initial segment of $\sigma$. 
\end{definition}
It should be noted that $T_G$ will have no terminal nodes. That is, every element of $T_G$ is an initial segment of another.

\begin{definition}
    Let  $G$ be a subgroup of $ S_\infty$. We define the degree of $T_G$ to be the join of the Turing degrees of 
    \begin{itemize}
        \item The domain of $T_G$ under some computable coding of $\N\lom$ in which the immediate predecessor relation (and thus $\sqsubset$) is decidable; and

        \item A branching function $Br:T_G\to \N\cup \{\infty\}$ such that $Br(\tau)$ is equal to the number of direct successors of $\tau$ in $T_G$. 
    \end{itemize}
    We denote the degree of $T_G$ as $\deg(T_G)$. 
\end{definition}
We will focus on groups where $T_G$ is finite branching, in which case the range of $Br$ will be a subset of $\N$. It should be noted that $\deg(T_G)$ is not invariant under group isomorphism in that it is possible to have $G\cong G'$ with $\deg(T_G)\neq \deg(T_{G'})$.

\begin{definition}
    Given a tree $T\subset \N\lom$, we define $[T]$ to be the set of all paths through $T$. We endow $[T]$ with the standard product topology in which the basic clopen sets are those of the form $\{f\in \N^\omega: \tau \sqsubset f\}$ for some $\tau \in T$.
\end{definition}

 It is clear that every element of $G$ is represented as a path through $T_G$. In particular, the function $i:G\to [T_G]$ defined by $$i(g)=g(0)g\iv(0)g(1)g\iv(1)\cdots$$ is an embedding. However, it is possible for there to be additional paths through $T_G$ that do not correspond to any element of $G$ in such a way. For example, consider the group  $G$ generated by $\{(0\,1),(2\,3),(4\,5),...\}$ where $(n\,m)$ denotes the permutation that swaps $n$ and $m$ and leaves everything else fixed. We see that $G$ is countable but $[T_G]$ is size continuum. The following proposition gives a simple topological condition for when a group $G$ corresponds nicely with $[T_G]$.

 \begin{proposition}
 \label{closed}
     Let  $G$ be a subgroup of $ S_\infty$. The map $i:G\to [T_G]$ is a bijection if and only if $i(G)$ is a closed subset of $[T_G]$. \qed
 \end{proposition}

 We say that $G$ is a $\textit{closed group}$ when $i(G)$ is closed. Thus, the subgroups of $S_\infty$ that can be represented as the paths through this type of tree are exactly the closed ones. Additionally, we say that $G$ is a \textit{compact group} if $i(G)$ is compact. 

 \begin{definition}
     A  topological group is called profinite if it is isomorphic to the inverse limit of an inverse system of discrete finite groups.
 \end{definition}
The following proposition yields a simple topological definition for profinite groups.

 \begin{proposition}[Folklore; see e.g.\ Theorem 3.7 from \cite{osserman}]
 \label{osserman theorem}
A topological group is profinite if and only if it is Hausdorff, compact, and totally disconnected.\qed
 \end{proposition}

\begin{definition}
    Given a subgroup $G$ of $ S_\infty$ and $n\in \N$, we define the orbit of $n$ under $G$ as $$\orb_G(n):=\{g(n)\in \N: g\in G \}.$$
\end{definition}
The following proposition is also folklore, but we give a brief proof.

\begin{proposition}
    Let $G$ be a subgroup of $ S_\infty$. The following are equivalent: 
    \begin{enumerate}
        \item[(1)] $G$ is compact, 
        \item[(2)] $G$ is closed and all orbits under $G$ are finite,
        \item[(3)] $G$ is profinite.
    \end{enumerate}
\end{proposition}

\begin{proof}
 Suppose that $G$ is compact. Since our topology is Hausdorff, we have that $G$ is closed. By Proposition \ref{closed}, $i(G)=[T_G]$.  Assume towards a contradiction that there is some $n\in \N$ with $\orb_G(n)$ infinite. Let $\{\tau_i\}_{i\in \N}$ be the (infinite) collection of all elements of $T_G$ of length $n+1$. Note that $\{\{f:\tau_i\sqsubset f\}\}_{i\in \N}$ is an open cover of $[T_G]=i(G)$ with no finite subcover, which contradicts that $G$ is compact. Hence, $(1)\implies (2)$.

 If $G$ is closed and all orbits in $G$ are finite, then it follows that $G$ is compact as a consequence of K\"onig's lemma. Hence, $(2)\implies (1)$. 

 The topology we have defined is Hausdorff and totally disconnected. If $G$ is compact, then $G$ is also profinite by Proposition \ref{osserman theorem}. Hence, $(3)\iff (1)$. \qed
\end{proof}

We have that all profinite subgroups of $S_\infty$ will have countably many orbits (all of which are finite). We fix an enumeration of these orbits as follows:

\begin{definition}
    Let $G$ be a profinite subgroup of $S_\infty$. Define $\{O_{G,i}\}_{i\in \N}$ so that $O_{G,0}=\orb_G(0)$ and $O_{G,n+1}$ is the orbit of the least natural number not in any $O_{G,m}$ with $m\leq n$.
\end{definition}
We can use these orbits to define finite approximations of $G$ up to the first $k\in \N$ many orbits.

\begin{definition}
\label{gs}
    Let $G$ be a profinite subgroup of $S_\infty$. Given $g\in G$ and $k\in \N$, define $g_k=g\upharpoonright \bigcup_{i\leq k}O_{G,i}$. Define $$G_k:=\{g_k:g\in G\}.$$ 
\end{definition}
We can also define the restriction of $G$ to only the $k$th orbit.

\begin{definition}
\label{G_k}
    Let $G$ be a profinite subgroup of $S_\infty$. Define $$H_k:= \{g\upharpoonright O_{G,k}: g\in G\}.$$
\end{definition}
Note that both $G_k$ and $H_k$ are finite groups for all $k\in \N$. Additionally, they are both uniformly computable given $T_G$.

\begin{definition}
    Let $G$ be a profinite subgroup of $S_\infty$. We say that $G$ has orbit independence if it is isomorphic to the Cartesian product of all $H_k$. That is, $$G\cong \prod_{k\in \N} H_k.$$
\end{definition}

It should be noted that the property of orbit independence is not preserved under isomorphism. In fact, every profinite group is isomorphic to a subgroup of $S_\infty$ without orbit independence. For an example, consider the groups $$G=\{1_G, (0\,1)(2\,3)\}$$ and $$G'=\{1_{G'},(0\,1)\}$$ where $1_G$ denotes the identity permutation. Note that $G\cong G' \cong\Z/2\Z $ and it is clear that $G'$ has orbit independence. However for $G$, we have $H_0\cong H_1\cong \Z/2\Z $  and $H_n$ is trivial for all $n>1$. Thus, $\prod H_k\cong \Z/2\Z \times \Z/2\Z \ncong G$. 

A profinite subgroup of $S_\infty$ is isomorphic to such a subgroup with orbit independence if and only if it is isomorphic to the direct product of finite groups. Since not all profinite groups can be expressed as such a direct product, some profinite groups have no orbit independent representation. This is illustrated by the following proposition:

\begin{proposition}
    Let $G$ be a (nontrivial) profinite subgroup of $S_\infty$. If $G$ is torsion free, then $G$ does not have orbit independence. 
\end{proposition}

\begin{proof}
    Suppose we had $G\cong \prod_{k\in \N} H_k$. Take the least $n$ such that $H_n$ is nontrivial and let $h$ be some non-identity element of $H_n$. Since all $H_k$ are finite, $h$ is a torsion element. Define the element $(h_k)_{k\in \N}\in G$ such that $h_n=h$ and $h_m$ is the identity when $m\neq n$. This would be a torsion element of $G$.  \qed
\end{proof}

Orbit independent profinite groups do arise naturally, including in the context of Galois theory. For example, consider the field $F=\Q(\sqrt{p_i})_{i\in \N}$ where $p_i$ denotes the $i$th prime number. We have that the absolute Galois group of $F$ is isomorphic to the direct product of countably many copies of the two element group, which is naturally embedded as an orbit independent subgroup of $S_\infty$.  

Our definition for the degree of $T_G$ is compatible with the notion of a recursively profinite group used by La Roche and Smith.

\begin{definition}[\cite{la roche,smith}]
\label{recursive def}
 A profinite group $P$ is called recursively profinite if there exists a uniformly computable sequence $\{P_n, \pi_n\}_{n\in \N}$ such that each $P_n$ is a finite group, each $\pi_n$ is a surjective homomorphism from $P_{n+1}$ to $P_n$, and $P$ is isomorphic to the inverse limit of the sequence.
     
 \end{definition}

\begin{proposition}
\label{compatible defs}
    A profinite group $P$ is recursively profinite if and only if it is isomorphic to a subgroup $G$ of $S_\infty$ with $T_G$ computable.
\end{proposition}

\begin{proof}
    Suppose that $P\cong G$ with $T_G$ computable. Defining $\nu_n: G_{n+1}\to G_n$ (with $G_n$ as in Definition \ref{G_k}) so that $\nu_n( g_{n+1})=g_n$, we get that $\{G_n,\nu_n\}_{n\in \N}$ is a uniformly computable sequence as required in definition \ref{recursive def} whose inverse limit is isomorphic to $P$.

    For the other direction, suppose that $\{P_n,\pi_n\}_{n\in \N}$ is as in definition $\ref{recursive def}$. For each $n\in \N$, let $N_n$ be a natural number such that $P_n$ is isomorphic to a subgroup of $S_{N_n}$ (the group of permutations of $\{0,...,N_n-1\}$). Define $f_0:P_0\to S_{N_0}$ such that $f_0$ is a group embedding. Given $f_n$, define $f_{n+1}:P_{n+1}\to S_{N_{n+1}}$ such that $f_{n+1}$ is a group embedding and ``respects'' $\pi_{n+1}$ in the sense that if $\pi_{n+1}(p_{n+1})=p_n$, then $f_{n+1}(p_{n+1})\upharpoonright N_n=f_n(p_n)$. Define $G$ to be the set of $g\in S_\infty$ such that for all $n\in \N$, there exists a $p\in P_n$ with $f_n(p)= g\upharpoonright N_n$. We have that $G\cong P$ and that $T_G$ is exactly the set $$\left\{\tau\in \N\lom: \left(\E n\in \N, p\in P_n, m<N_n\right)[\tau = f_n(p)(0)f_n(p)(1)\cdots f_n(p)(m)]\right\}.$$ It is clear that the domain and branching function of $T_G$ are computable, hence $T_G$ is computable. \qed

\end{proof}

\section{Complexity of Theories}
\label{main results}
We now consider the complexity of the existential theory of a profinite subgroup of $S_\infty$. To do so, we must first establish a few lemmas. 

\begin{definition}
     A positive formula is a first order formula that can be expressed using only the logical connectives $\wedge$ and $\vee$  without the use of negation symbols. A negative formula is the negation of a positive formula.

\end{definition}
It should be noted that quantifiers are allowed in positive/negative formulas and that many formulas are neither positive nor negative. The following lemma illustrates a well known property of positive formulas.

\begin{lemma}
\label{pos preserved}
    Positive formulas are preserved under surjective homomorphisms. That is, if $\alpha^+$ is a positive formula and $f:\A\to \B$ is a surjective homomorphism, then $\A\models \alpha^+(\bar a)$ implies that $\B\models \alpha^+(f(\bar a))$. \qed
\end{lemma}

\begin{lemma}
\label{positive goes down}
Let $G$ be a profinite subgroup of $S_\infty$ and let $\alpha^+$ be a positive formula in the language of groups. If $k<l$, then $$G_l\models \alpha^+(\bar g_l) \implies G_k\models \alpha^+(\bar g_k)$$ for any $\bar g \in G\lom$, where $\bar g_l$ and $\bar g_k$ are as in Definition \ref{gs}. 

\end{lemma}

\begin{proof}
   The natural projection $\pi: G_l \to G_k$ defined such that $\pi(g_l)=g_k$ is a surjective homomorphism. Hence, the result follows from Lemma \ref{pos preserved}. \qed
\end{proof}

\begin{corollary}
\label{cor}
    Let $G, k$ and $l$ be as in the previous lemma. If $\alpha^-$ is a negative formula, then 
    $$G_k\models \alpha^{-}(\bar g_k) \implies  G_l\models \alpha^{-}(\bar g_l) $$ for any $\bar g\in G\lom$. \qed
\end{corollary}

\begin{lemma}
\label{no orb ind}
   Let $G$ be a profinite subgroup of $S_\infty$.  If $\alpha$ is  quantifier free,  then $G\models \alpha(\bar g)$ if and only if  $G_k\models \alpha(\bar g_k) $ for all sufficiently large  $k\in \N$.
\end{lemma}

\begin{proof}
    For the base case, let $\alpha$ be atomic. We have that $\alpha(\bar x)\equiv W(\bar x)=1$ where $W(\bar x)$ is a word over $\{x,x\iv: x\in \bar x\}$ and $1$ is the group identity symbol. Clearly, if $W(\bar g)=1_G$ then $W(\bar g_k)=1_{G_k}$ for all $k$. On the other hand, if $W(\bar g_k)=1_{G_k}$ for sufficiently large $k$ then we have by the previous lemma that $W(\bar g_k)=1_{G_k}$ for all $k\in \N$. Thus given any $n\in \N$ and a large enough $l$ such that $n\in \dom(\bar g_l)$, we have that $W(\bar g_l)$ maps $n$ to $n$. Hence, $W(\bar g)=1_G$. 

    \paragraph{Negative Step:} Let $\alpha\equiv\neg\beta$ with $\beta$ atomic. Suppose $G\models \neg \beta(\bar g)$. This gives that there is a $k$ such that $G_k\models \neg \beta(\bar g_k) $. By the previous corollary, we must have that $G_l\models \neg \beta(\bar g_l) $ for all $l\geq k$ and thus for all sufficiently large $l$.

    Now suppose that $G_k\models \neg \beta(\bar g_k)$ for sufficiently large $k$. There must only be finitely many $k$ such that $G_k\models \beta(\bar g_k)$. Thus, from the base case, we have that $G\models \neg \beta(\bar g)$.

    \paragraph{Conjunctive/Disjunctive Step:} If the statement holds for $\beta_1$ and $\beta_2$, then it is clear that it holds for $\beta_1\,\&\,\beta_2$ as well. If the statement holds for $\beta_1$ or for $\beta_2$, then it is clear that it holds for $\beta_1\oor\beta_2$ as well. \qed

\end{proof}

\begin{lemma}
\label{and k}
    Let $G$ be a profinite subgroup of $S_\infty$ with orbit independence. Let $\alpha$ be an existential sentence in the language of groups. We have that $G\models \alpha$ if and only if $G_k\models \alpha$ for some $k\in \N$.
\end{lemma}

\begin{proof}
    We have that $\alpha\equiv \E \bar x\beta$ where $\beta$ is a quantifier free formula. If $G\models \alpha $ then there is some $\bar g\in G\lom$ such that $G\models \beta(\bar g)$. Thus, Lemma \ref{no orb ind} gives that there is a $k$ with $G_k\models \beta(\bar g_k)$ and so $G_k\models \alpha$. 

     Now suppose that $G_k\models \alpha$. Define an embedding $f:G_k\to G$ such that $f(\bar \gamma)_k=\bar \gamma$ and $f(\bar \gamma)$ is just the identity on all orbits $O_{G,l}$ with $l>k$. Note that since $G$ has orbit independence, we will in fact have that $f(\bar \gamma)\in G$ and so $f$ is a well defined embedding. By the \L{}o\'s-Tarski preservation theorem,  existential formulas are preserved under embedding and so $G\models \alpha$.  \qed

    \begin{theorem}
\label{ex theory with orb in}
    Let $G$ be a profinite subgroup of $S_\infty$ with orbit independence. The existential theory of $G$ is $\Sigma^0_1$ relative to $\deg(T_G)$. 
\end{theorem}

\begin{proof}
    Let $\alpha$ be an existential sentence. By the previous lemma, $G\models \alpha$ if and only if $$(\E k)[G_k\models \alpha]$$ which is $\Sigma^0_1$ relative to $\deg(T_G)$. \qed
\end{proof}

    \end{proof}
The following proposition gives that the above theorem is sharp. 

\begin{proposition}
    There exists a profinite  subgroup $G$ of $S_\infty$ with orbit independence such that $T_G$ is computable and the existential theory of $G$ is $\Sigma^0_1$ complete. 
\end{proposition}

\begin{proof}
    By the previous theorem, we need only build a $G$ with $T_G$ computable such that the existential theory codes $\0'$. Define the formula $\alpha_n$ for all $n\in \N$ by $$\alpha_n:= (\E x)[x\neq 1 \,\&\, x^{p_n}=1]$$ where $1$ is the identity element and $\{p_n\}_{n\in \N}$ is the sequence of all primes. We build $G$ such that $G\models \alpha_n$ if and only if $n\in \0'$.  

    \paragraph{Construction} 
    \paragraph{Stage $0$:} Define $O_{G,0}=\{0\}$ and define $H_0$ to be the trivial group. 

    \paragraph{Stage $s$:} Let $N_s\in \N$ be the least not in any $O_{G,i}$ with $i<s$. Find the least $e\leq s$ such that $\Phi_{e,s}(e)\downarrow$ and $G_{s-1}\models \neg\alpha_e$. If no such $e$ exists, define $O_{G,s}=\{N_s\}$ and $H_s$ to be the trivial group. If there is such an $e$, define $O_{G,s}=\{N_s, N_s+1,...,N_s+p_e -1\}$ and define $H_s$ to be  cyclic  on $O_{G,s}$.  
    \paragraph{Verification}
    Since each $G_s$ is computable, it is clear that the tree $T_G$ is computable. If $n\notin \0'$, then no $H_s$ will be of size $p_n$. Thus, no element has order $p_n$, which gives $G\models \neg \alpha_n$. If $n\in \0' $, then there will come a stage $t$ in which $n$ is the least such that $\Phi_{n,t}(e)\downarrow$ and there is currently no $H_s$ of size $p_n$. We will then make $H_t$ cyclic and of size $p_n$ which will assure that $G\models \alpha_n$.
    \qed
\end{proof}

\begin{theorem}
\label{sigma 2}
    Let $G$ be any profinite subgroup of $S_\infty$ (not necessarily with orbit independence). The existential theory of $G$ is $\Sigma^0_2$ relative to $\deg(T_G)$.
\end{theorem}

\begin{proof}
    Suppose $\alpha=\E\bar x\beta$ with $\beta$ quantifier free. By Lemma \ref{no orb ind}, given $\bar g\in G\lom$ we have that $G\models\beta(\bar g)$ if and only if $G_k\models \beta(\bar g_k)$ for all but finitely many $k\in \N$. Let $T_\beta$ be the subset of $T_G$ defined by $$T_\beta:=\{ \tau \in T_G: G_{l(\tau)}\models \beta( \tau))\}$$ where $l(\tau)$ is defined as the natural number such that $ \tau\in G_{l(\tau)}$. Note that $T_\beta$ is computable given $T_G$. We have that $G\models \alpha$ if and only if $$(\E \bar\tau\in T_G\lom)(\forall k\geq l(\bar \tau) )\left[\bigvee_{\bar \sigma\in G_k\lom}\left( \bar\tau  \sqsubseteq \bar \sigma\,\&\, \bigwedge_{\bar \tau\sqsubseteq\bar \rho \sqsubseteq \bar \sigma} \bar\rho\in T_\beta\right)\right] $$ which is $\Sigma^0_2$ relative to $\deg(T_G)$ (recall that $\deg(T_G)$ computes the branching function $T_G$, and can thus compute the elements of each $G_k$).   \qed
\end{proof}
The following proposition gives that the above theorem is sharp. 
\begin{proposition}
\label{sigma 2 complete}
    There exists a profinite subgroup $G$ of $S_\infty$ (without orbit independence) with $T_G$ computable such that the existential theory of $G$ is $\Sigma^0_2$ complete.
\end{proposition}

\begin{proof}
    Recall that the set $\Fin=\{e\in \N: |W_e|<\infty\}$, where $W_e$ is the domain of $\Phi_e$, the $e$th Turing program, is $\Sigma^0_2$ complete (see \cite{soare}, Theorem 4.3.2). Let  $\{p_n\}_{n\in \N}$ be the sequence of all primes. Given $n\in \N$, define the formula $$\alpha_n:= (\E x)[x\neq 1 \,\&\, x^{p_n}=1].$$ We construct $G$ so that $$G\models \alpha_n \iff n\in \Fin.$$ This, along with Theorem \ref{sigma 2}, ensures that the existential theory of $G$ is $\Sigma^0_2$ complete. 

   We construct $G$ in stages, defining $G_s$ at stage $s$ of the construction. At stages of the form $s=\langle n,m\rangle$, we work toward making sure that $G$ will model $\alpha_n$ just if $W_n$ is finite. Specifically, if $|W_{n,m+1}|>|W_{n,m}|$ then we make sure that if $g\in G_{s-1}$ with $g\neq 1_{G_{s-1}}$, then any  $g'\in G_s$ with $g\sqsubset g'$ has $g'^{p_n}\neq 1_{G_s}$. We also create a new element not equal to $1_{G_s}$ that is of order $p_n$. If $W_{n,m+1}=W_{n,m}$, then we define $H_s$ to be the trivial group (thus if $g\sqsubset g'$ with $g\in G_{s-1}$ and $g'\in G_s$, then $g'^{p_n}=1_{G_{s}}$ if and only if $g^{p_n}=1_{G_{s-1}}$).

   \paragraph{Construction} Define a bijection $\langle\rangle:\N^2\to \N$ such that $0=\langle 0,0\rangle$ and $\langle n,m\rangle <\langle n,m+1\rangle $ for all $n,m\in \N$. Define $l_0=0$. For all $s>0$, define $l_s$ to be the least natural number not in $O_{G,s-1}$. 

   \paragraph{Stage $0=\langle 0,0\rangle$:} Define $O_{G,0}=\{0,1\}$  and $H_0=G_0=\{(0)(1), (0\,1)\} $.  
   \paragraph{Stage $s=\langle n,0\rangle$ with $n>0$:} Define $O_{G,s}=\{l_s,l_s +1,...,l_s + p_n-1\}$. Define $H_s$ to be the cyclic group on $O_{G,s}$. Define $G_s=\{g^\frown h: g\in G_{s-1}, h\in H_s$\}. 

   \paragraph{Stage $s=\langle n,m\rangle$ with $m>0$:} Check if $|W_{n,m}|>|W_{n,m-1}|$. 
   \begin{itemize}
       \item If no, then define $O_{G,s}=\{l_s\}$ and $G_s=\{g^\frown (l_s):g \in G_{s-1}\}$. Note, this gives that $H_s$ is the trivial group.
       \item If yes, then take $N$ so that this is the $N$th time that $|W_{n,x}|>|W_{n,x-1}|$. That is, define $$N=1+|\{x\in \N:0<x<m \,\&\, |W_{n,x}|>|W_{n,x-1}|\}|.$$ Define $O_{G,s}=\{l_s,l_{s}+1,...,l_{s} + p_n^{N+1}-1\}$ and define $H_s$ to be the cyclic group on $O_{G,s}$. Define $t$ to be the  stage that we had added an orbit of size $p_n^N$.  We define $$G_s=\{g_i^\frown h_i: 0\leq i<p_n^N \, \& \, h_i\in H_s\, \&\, g_i\in G_{s-1}\text{ with } g_i(l_t)-l_t\equiv h_i(l_s)-l_s \mod p_n^N \}.$$

   \end{itemize}

   For example, suppose we are at stage $1=\langle 0,1 \rangle$. If $W_{0,1}=W_{0,0}=\0$, then we will have $$G_1=\{(0)(1)(2), (0\,1)(2)\}.$$ If $W_{0,1}\neq \0$, then we will have $$G_1=\{(0)(1)(2)(3)(4)(5),(0)(1)(2\, 4)(3\,5),(0\,1)(2\,3\,4\,5),(0\,1)(2\,5\,4\,3)\}.$$

   \paragraph{Verification}
   Since each $G_s$ is computable, it is clear that the tree $T_G$ is computable. Thus we need only show that $G\models \alpha_n$ if and only if $n\in \Fin$. 

   \begin{lemma}
   \label{lemmaa}
       Let $s=\langle n,m\rangle$ with $|W_{n,m}|>|W_{n,m-1}|$. If $g\in G_s$ with $g^{p_n}=1_{G_s}$, then $g\upharpoonright{l_s}=1_{G_{s-1}}.$
   \end{lemma}

   \begin{proof} Let $N$ and $t$  be defined as they were at stage $s=\langle n,m\rangle$ of the construction. Note that $(g_i^\frown h_i)^{p_n}=1_{G_s}$ if and only if $h_i^{p_n}=1_{H_s}$ and $g_i^{p_n}=1_{G_{s-1}}$. Since $h_i^{p_n}=1_{H_s}$, we must have that $h_i(l_s)\equiv 0\mod p_n^N$.  This gives that $g_i(l_t)=l_t$, and so $g_i$ is the identity permutation when restricted to $O_{G,t}$. Similarly, we will get that $g_i$ is the identity permutation on $O_{G,r}$ for all $r$ of the form $r=\langle n,x\rangle$. In order for $g_i^{p_n}=1_{G_{s-1}}$, we must also have that $g_i$ is the identity permutation on $O_{G,r}$ for all $r$ that are not of the form $\langle n,x\rangle$ as all of these $O_{G,r}$ will either be of size 1, or of a size not divisible by $p_n$. Hence, we get that $g_i=g\upharpoonright l_s=1_{G_{s-1}}$. \qed \end{proof}

   If $n\notin \Fin$, then there will be infinitely many stages $s$ of the form $s=\langle n,m\rangle$ with $|W_{n,m}|>|W_{n,m-1}|$. Let $g\in G$. If $g^{p_n}=1$, then by Lemma \ref{lemmaa} we have $g_{s-1}=g\upharpoonright l_{s-1}=1_{G_{s-1}}$ for each such $s$. However, if $g\in G$ was a witness to $\alpha_n$ then we would have by Lemma \ref{no orb ind} that $G_k\models g_k\neq 1_{G_k}\,\&\, g_k^{p_n}=1_{G_k}$ (where $g_k=g\upharpoonright l_k$) for all but finitely many $k$, which is a contradiction. Hence, there is no witness to $\alpha_n$ in $G$ and so $G\models \neg \alpha_n$. 

   Now suppose that $n\in \Fin$. We have that there is a least natural number $m$  such that $W_n$ gains no new elements after stage $m$. This  gives that for $s=\langle n,m\rangle $, $|H_s|$ is a multiple of $p_n$, but no $H_x$ with $x>s$ will have $|H_x|$ divisible by $p_n$. Note that by our instructions, there will be an element $g\in G_s$ that is not the identity, but is of order $p_n$. Since we will never have $|W_{n,x}|>W_{n,x-1}$ for any $x>s$, we will have that there is an element of $G$ that is equal to $g$ on $G_s$, and is equal to the identity on all orbits higher than that of $G_s$. This element will be a witness to $\alpha_n$. \qed
\end{proof}

So far we have only considered existential theories. We conclude by now expanding to entire first order theories for subgroups with orbit independence, which we show to be $\Delta^0_2$ relative to $\deg(T_G)$ as a consequence of the following theorem of Feferman and Vaught.

\begin{theorem}[Theorem 6.6 from \cite{FV}]
\label{first fv corollary}
Given any first order $\LL$-sentence $\phi$, we can compute $n\in \N$ such that for every family $\{\A_i:i\in I\}$ of $\LL$-structures there exists $J\subseteq I$ with $|J|\leq n$ such that if $\prod_{i\in I}\A_i\models \phi$, then $\prod_{i\in J'}\A_i\models \phi $ for all $J'$ with $J\subseteq J'\subseteq I$. \qed

\end{theorem}

\begin{corollary}
\label{FV corollary}
    Let $G$ be a profinite subgroup of $S_\infty$ with orbit independence. Let $\alpha$ be any first order sentence in the language of groups. We have $G\models \alpha$ if and only if $G_k\models \alpha$ for all sufficiently large $k\in \N$.
\end{corollary}

\begin{proof}
    Since $G$ has orbit independence, we have that $G\cong \prod_{i\in \N} H_i$. If $G\models \alpha$, then Theorem \ref{first fv corollary} gives that there is some finite $J\subset \N$ such that if  $J'\supseteq J$, then $\prod_{i\in J'}H_i\models \alpha$. Thus, for all $k\geq \max(J)$ we have $G_k\models \alpha$. For the other direction note that if $G\nvDash \alpha$, then $G\models \neg \alpha$ and so the same reasoning  gives that $G_k\models \neg \alpha$ for sufficiently large $k$. \qed
\end{proof}
It should be noted that the result in Corollary \ref{FV corollary} does not hold when the condition of orbit independence is dropped, as shown by the following proposition.

\begin{proposition}
    There exists a profinite subgroup of $S_\infty$ (without orbit independence) and existential formula $\alpha$ such that $G_k\models \alpha$ for all $k$, but $G\models \neg \alpha$. 
\end{proposition}

\begin{proof}
    Let $\alpha$ be the sentence that states that there is an element of order $2$. That is, $\alpha=(\E x)[x\neq 1 \, \& \, x^2= 1]$. We build a $G$ such that each $G_k$ has an element of order $2$, but $G$ has no such element. The $G$ we build will be isomorphic to the (additive) group of $2$-adic integers. 
    
    Define $O_{G,0}=\{0,1\}, O_{G,1}=\{2,3,4,5\}, O_{G,2}=\{6,7,8,9,10,11,12,13\} $ and so on.  Define $H_k$ to be the cyclic group on $O_{G,k}$ for all $k\in \N$. Define $G_0=\{(0)(1),(0\,1)\}$ and 
    define $$G_k=\{g^\frown h: g\in G_{k-1}, h\in H_k, \ord(h)\equiv \ord(g)\mod 2^{k} \}$$ for all $k>0$ (where $\ord(x)$ denotes the order of $x$). Note, each $G_k$ has exactly one element of order $2$. However, every non-identity element of $G$ will have infinite order. \qed
\end{proof}

\begin{theorem}
\label{delta 2}
    Let $G$ be a profinite subgroup of $S_\infty$ with orbit independence. The first order theory of $G$ is $\Delta^0_2$ relative to $\deg(T_G)$. 
\end{theorem}

\begin{proof}
    Let $\Th(G)$ denote the first order theory of $G$. By Corollary \ref{FV corollary} we have that $\alpha\in \Th(G)$ if and only if $$(\E l)(\forall k>l)[G_k\models \alpha] $$ which is  $\Sigma^0_2$ relative to $\deg(T_G)$. On the other hand, we have $\alpha\notin \Th(G)$ if and only if $$(\E l)(\forall k>l)[G_k\models\neg \alpha]$$ which is  $\Sigma^0_2$ relative to $\deg(T_G)$. Hence, both $\Th(G)$ and its complement are $\Sigma^0_2$ relative to $\deg(T_G)$ and so $\Th(G)$ is $\Delta^0_2$ relative to $\deg(T_G)$. \qed
\end{proof}

This draws a strong distinction between the complexity of theories of profinite groups with and without orbit independence. Note that by the proof of Proposition \ref{sigma 2 complete} it is possible for just the existential theory of $G$ to be $\Sigma^0_2$ complete relative to $\deg(T_G)$ when $G$ does not have orbit independence. However, the entire first order theory of $G$ will be $\Delta^0_2$ relative to $\deg(T_G)$ when $G$ has orbit independence. In contrast, bounds on the complexity of the first order theory of profinite $G$ without orbit independence remain unknown. The author hopes to investigate this question in the near future.

\section{Elementarity}
\label{Elementarity}
Last, we examine to what extent certain well defined countable subgroups form elementary subgroups of profinite groups. In particular, we will look at the subgroup of all computable elements as well as the subgroup of elements that are computable in some \emph{Turing ideal}.  

\begin{definition}
     A Turing ideal is a collection $I$ of Turing degrees such that given any $\aaa,\bbb \in I$, we have $\aaa\oplus \bbb\in I$ as well as $\ccc\in I$ for all  $\ccc\leq \aaa$.  
\end{definition}

\begin{definition}
    Given a subgroup $G$ of $S_\infty$ and a Turing ideal $I$ with $\deg(T_G)\in I$, we define $G_I$ to be the elements of $G$ whose degrees are in $I$. That is, $$G_I:=\{g\in G: \deg(g)\in I\}.$$ 
\end{definition}
Note, $G_I$ is indeed a subgroup of $G$. By downward closure $\bold 0\in I$ and so the identity is in $G_I$. For any $g$, $\deg(g\iv)=\deg(g)$ and so $g\iv \in G_I$ whenever $g\in G_I$. And given $g_1,g_2\in G_I$, we have $\deg(g_1 \circ g_2)\leq \deg(g_1)\oplus \deg(g_2)$ and thus $g_1\circ g_2\in G_I$.

Substructures containing only the computable elements of an uncountable structure have been previously examined. For example, the collection of computable real numbers $\R_{\{\bold 0\}}$ was shown to be an ordered field by Rice in \cite{Rice}. More recently Korovina and Kudinov  \cite{k and k} determined the degree spectrum of this ordered field. Given any Turing ideal $I$, the collection $\R_I$ of all $I$-computable reals also forms an ordered field. These substructures of $\R$ are examples of the much studied notion of an elementary substructure. 

\begin{definition}
    Let $\A$ be a substructure of $\B$ and let $\Gamma$ be a class of formulas. We say that $\A$ is a $\Gamma$-elementary substructure if for all formulas $\gamma\in \Gamma$ and tuples $\bar a\in \A$, $$\A \models \gamma(\bar a) \iff \B\models \gamma(\bar a).$$ We express this as $$\A\preceq_\Gamma \B.$$ If this holds for all first order formulas $\gamma$, then we simply say that $\A$ is an elementary substructure of $\B$ and write $$\A\preceq \B.$$
\end{definition}
It is not difficult to show that every $\R_I$ is a real closed field, and thus an elementary substructure of $\R$ by model completeness. In this section we look for similar results for profinite groups by investigating what conditions must be placed on $G$, $I$, and $\Gamma$ in order to have $G_I\preceq_\Gamma G$. It is of particular interest to know when we will have  $G_I\preceq G$. 

The following theorem shows that we will not always have $G_I\preceq G$.

\begin{theorem}
\label{not always elementary}
    There exists a profinite subgroup $G$ of $S_\infty$ such that $\deg(T_G)=\bold 0$, but $G_{\{\bold 0\}}$ is not an existential elementary subgroup of $G$ (where $\{\bold 0\}$ is the Turing ideal whose only element is $\bold 0$). That is, $$G_{\{\bold 0\}} \npreceq_\E G$$ where $\E$ is the class of existential first order formulas in the language of groups. 
\end{theorem}

\begin{proof}
     To prove this we build $G$ along with an element $h\in G$ such that $T_G$ and $h$ are computable, $h$ has a square root in $G$, but $h$ has no computable square root in $G$. This will give that $G\models (\E \bar x)[x^2=h]$ but $G_{\{\bold 0\}}\models \neg (\E \bar x)[x^2=h]$. 

     We construct $G$ and $h$ in stages. By \emph{level} $m$ of the tree $T_G$, we mean the layer of the tree that shows where $m\in \N$ is being mapped to by the elements of $G$. The tree $T_G$ will be two branching at all levels $m_n$ (where $\{m_n\}_{n\in \N}$ will be a computable sequence in $\N$) and either one or two branching at all  other levels. Given $\sigma, \tau \in T_G$, we say that $\sigma$ is to the left of $\tau$ if $\tau$ comes before $\sigma$ lexicographically. We will assure that no computable path through $T_G$ will give a square root of $h$. To do this we will have that for each $n\in \N$ \begin{itemize}
        \item If $\Phi_n(m_n)=0$, then no path that goes left at level $m_n$ will give a square root of $h$; and 
        \item If $\Phi_n(m_n)=1$, then no path that goes right at level $m_n$ will give a square root of $h$.
    \end{itemize}
In the case where $\Phi_n(m_n)\neq 0,1$ or $\Phi_n(m_n)\uparrow$, it will not matter which direction a path goes at level $m_n$. We will build $G$ orbit by orbit, defining $G$ as the limit of all $G_k$.

\paragraph{Construction:} For all $n\in \N$, define $l_n=n(n+1)/2$. Stages of the form $l_n + e+1$ (with $e+1<l_{n+1}-l_n$) will be dedicated to making sure that $\Phi_e$ does not compute a square root of $h$. For all $n\in \N$, define $m_n$ to be the smallest natural number not in $O_i$ for any $i\leq n$.

\paragraph{Stage $0$:} Define $O_0=\{0,1\}$, $G_0=\{(0)(1),(0\,1)\}$, and $h_0=(0)(1)$.

\paragraph{Stage $s+1=l_n$:} Define $O_{s+1}=\{m_s,m_s+1\}$, $$G_{s+1}=\{g^\frown (m_s)(m_s+1), g^\frown(m_s\,\,\, m_s+1): g\in G_s\},$$ and $h_{s+1}=h_s^\frown (m_s)(m_s+1)$. 

\paragraph{Stage $s+1=l_n+ e+1$:} What we do at this stage will depend on if $\Phi_e(m_e)$ has halted yet.  
\begin{itemize}
    \item If $\Phi_{e,s}(m_e)\uparrow$ or $\Phi_{e,s}(m_e)\downarrow\neq 0,1$, then define $O_{s+1}=\{m_s,m_s+1\}$, $$G_{s+1}=\{g^\frown (m_s)(m_s+1), g^\frown(m_s\,\,\, m_s+1): g\in G_s\},$$ and $h_{s+1}=h_s^\frown (m_s)(m_s+1)$. 
    \item Otherwise, we expand $G$ so that it will act as $\Z/4\Z$ at this level. Define $O_{s+1}=\{m_s,m_s+1,m_s+2,m_s+3\}$. We can think of the four permutations $(m_s)(m_s+1)(m_s+2)(m_s+3)$, $(m_s\,\,\, m_s+1\,\,\,m_s+2\,\,\,m_s+3)$, $(m_s\,\,\,m_s+2)(m_s+1\,\,\,m_s+3)$, and $(m_s\,\,\,m_s+3\,\,\,m_s+2\,\,\,m_s+1)$ as the elements $0,1,2,$ and $3$ in $\Z/4\Z$ respectively. We will have that for every $g\in G_s$, $g$ will be extended either by $(m_s)(m_s+1)(m_s+2)(m_s+3)$ and $(m_s\,\,\,m_s+2)(m_s+1\,\,\,m_s+3)$ or by $(m_s\,\,\, m_s+1\,\,\,m_s+2\,\,\,m_s+3)$ and $(m_s\,\,\,m_s+3\,\,\,m_s+2\,\,\,m_s+1)$. In particular, we will extend $g$ by $(m_s)(m_s+1)(m_s+2)(m_s+3)$ and $(m_s\,\,\,m_s+2)(m_s+1\,\,\,m_s+3)$ if and only if it goes left at level $m_e$.   That is, \begin{align*}
    G_{s+1}=\{&g^\frown (m_s)(m_s+1)(m_s+2)(m_s+3), g^\frown (m_s\,\,\,m_s+2)(m_s+1\,\,\,m_s+3),\\
    &g'^\frown(m_s\,\,\, m_s+1\,\,\,m_s+2\,\,\,m_s+3), g'^\frown(m_s\,\,\,m_s+3\,\,\,m_s+2\,\,\,m_s+1\,\,\,):\\
    &g,g'\in G_s \,\&\, g\upharpoonright O_e=(m_e)(m_e+1), g'\upharpoonright O_e=(m_e\,\,\, m_e+1)\}.
    \end{align*}
    If $\Phi_{e,s}(m_e)=0$, define $h_{s+1}=h_s^\frown(m_s\,\,\, m_s+2)(m_s+1\,\,\,m_s+3)$. If $\Phi_{e,s}(m_e)=1$, define $h_{s+1}=h_s^\frown(m_s)(m_s+1)(m_s+2)(m_s+3)$. 
\end{itemize}

\paragraph{Verification:} Since each $G_k$ is computable, it is clear that the tree $T_G$ is computable. Similarly, since each $h_k$ is computable, so is $h=\lim_k h_k$. 

Suppose towards a contradiction that $h$ had some computable square root $g$. There would have to be some $\Phi_e$ such that $\Phi_e(m_n)=0$ for all $n$ such that the path through $T_G$ defining $g$ goes left at level $m_n$, and $\Phi_e(m_n)=1$ otherwise.  Suppose we had such a $\Phi_e$. There will come some stage $s+1=m_n+ e+1$ for some $n$ with $s$ large enough so that $\Phi_{e,s}(m_e)\downarrow$. 

\begin{itemize}
    \item If $\Phi_{e,s}(m_e)=0$, then we will have that $g(m_e)=m_e$. However, our instruction for stage $s+1$ will make it so all paths $p\in [T_G]$ with $p(m_e)=m_e$ have $p(m_s)=m_s$ or $p(m_s)=m_s+2$. Our instructions will also have $h(m_s)=m_s+2$. Thus no such $p$, including $g$, can be a square root of $h$.
    \item If $\Phi_{e,s}(m_e)=1$, then we will have that $g(m_e)=m_e+1$. However, our instruction for stage $s+1$ will make it so all paths $p$ with $p(m_e)=m_e+1$ have $p\upharpoonright O_{s+1}=(m\,\,\, m+1\,\,\, m+2\,\,\, m+3)$ or $p\upharpoonright O_{s+1}=(m\,\,\,m+3\,\,\,m+2\,\,\,m+1)$. Our instructions will also have $h(m_s)=m_s$. Thus no such $p$, including $g$, can be a square root of $h$.
\end{itemize}
Hence, $g$ is not a square root of $h$ and we get that $h$ has no computable square root. 

We need only show that $h$ does indeed have some square root in $G$. Consider the path $p$ that goes right at all levels $m_n$ such that $\Phi_n(m_n)=0$, and goes left at all other levels. Note that $$h\upharpoonright O_{s+1}=\begin{cases}
    (m_s\,\,\,m_s +2)(m+1\,\,\,m+3) &; s+1=l_n+e+1 \,\& \, \Phi_{e,s}(m_e)=0\\
    1_{S_{O_{s+1}}} &; \text{otherwise}
    
\end{cases}$$  
and $$p\upharpoonright O_{s+1}=\begin{cases}
    (m_s\,\,\,m_s +1\,\,\,m+2\,\,\,m+3) &; s+1=l_n+e+1 \,\& \, \Phi_{e,s}(m_e)=0\\
    1_{S_{O_{s+1}}} &; \text{otherwise.}\end{cases}$$
    Hence, we have $p^2=h$. \qed
\end{proof}

It should be noted that the group $G$ used in the above proof does not have orbit independence and that the Turing ideal $\{\bold 0\}$ is not what is called a Scott ideal (defined below). As we will see from the next two propositions, if $G$ has orbit independence or if $I$ is a Scott ideal then we will have $G_I\preceq_\E G$.

\begin{proposition}
    Let $G$ be a profinite subgroup of $S_\infty$ with orbit independence. Given any Turing ideal $I$ with $\deg(T_G)\in I$, we must have $G_I\preceq_\E G$. 
\end{proposition}

\begin{proof}
    We must show that given any quantifier free formula $\alpha$ and any tuple $\bar a\in G_I\lom$, $$G_I\models (\E\bar x)\alpha(\bar x,\bar a) \iff G \models(\E\bar x)\alpha(\bar x,\bar a).$$ Note, the left to right direction is trivial. Thus we need only show that if $G\models (\E\bar x)\alpha(\bar x,\bar a)$, then there exists a witness $\bar g$ whose degree is in $I$. 

    Suppose we have $G\models (\E \bar x)\alpha(\bar x,\bar a)$. Let $\alpha\equiv \bigvee_i \alpha_i$ where each $\alpha_i$ has the form $$\alpha_i(\bar x,\bar y)=\bigwedge_j\alpha_{i,j}^+(\bar x,\bar y) \,\&\, \alpha_{i,j}^-(\bar x,\bar y)$$ where each $\alpha_{i,j}^+$  is a (quantifier free) positive formula and each $\alpha_{i,j}^-$ is a (quantifier free) negative formula. Fix an $i$ such that $G\models (\E \bar x)\alpha_i(\bar x,\bar a)$. By Lemma \ref{no orb ind}, we have that $G_k\models (\E \bar x)\alpha_i(\bar x,\bar a_k)$ for all sufficiently large $k\in \N$. Thus there exists some $L\in \N$ such that for all $k\geq L$, there is some $\bar\gamma_k\in G_k\lom$ such that $G_k\models \alpha_i(\bar \gamma_k,\bar a_k)$. Note that since all elements of $\bar a$ have degree in $I$ and $\deg(T_G)\in I$, we have that there is such a sequence $\{\gamma_k\}_{L\leq k\in \N}$ that is computable in $I$. 

    Define $\bar g$ such that $\bar g_L=\bar \gamma_L$ and  $\bar g_k\upharpoonright O_{G,k}= \bar \gamma_k\upharpoonright O_{G,k}$ for all $k>L$. Since $G$ has orbit independence we get that $\bar g\in G\lom$. Additionally since $\{\bar \gamma_k\}$ is computable in $I$, the degree of $\bar g$ is in $I$ as well and so $\bar g\in G_I\lom$. Note that since $G_L\models\bigwedge_j \alpha^-_{i,j}(\bar \gamma_L,\bar a_L)$, we have that $G_I\models \bigwedge_j \alpha_{i,j}^-(\bar g,\bar a)$. Additionally since $G_k\models \bigwedge_j \alpha_{i,j}^+(\bar \gamma_k,\bar a_k)$ for all $k\geq L$, we have that $G_I\models \bigwedge \alpha_{i,j}^+(\bar g,\bar a)$. Hence, $G_I\models \alpha_i(\bar g,\bar a)$ and so  $G_I\models (\E \bar x)\alpha(\bar x,\bar a)$. \qed
\end{proof}

\begin{definition}
    A Scott ideal is a Turing ideal $I$ such that given any $\aaa\in I$, there exists $\bbb\in I$ that is PA relative to $\aaa$.
\end{definition}
Scott ideals will be of particular interest to us because of the following well known proposition. 
\begin{proposition}

   A degree $\bbb$ is PA relative to a degree $\aaa$ if and only if $\bbb$ computes a path through every $\aaa$-computable finite branching infinite tree. \qed
\end{proposition}

\begin{proposition}
    Let $G$ be any profinite subgroup of $S_\infty$. If $I$ is a Scott ideal with $\deg(T_G)\in I$, then $G_I\preceq_\E G$. 
\end{proposition}

\begin{proof}
  We must show that given any quantifier free formula $\alpha$ and any tuple $\bar a\in G_I\lom$, $$G_I\models (\E\bar x)\alpha(\bar x,\bar a) \iff G \models(\E\bar x)\alpha(\bar x,\bar a).$$ We again have that the left to right direction is trivial. Thus we need only show that if $G\models (\E\bar x)\alpha(\bar x,\bar a)$, then there exists a witness $\bar g$ whose degree is in $I$. 

    Suppose that $G\models (\E \bar x)\alpha(\bar x,\bar a)$. Fix some $\bar g\in G\lom$ such that $G\models \alpha(\bar g,\bar a)$. By Lemma \ref{no orb ind} we have that there exists an $L\in \N$ such that $G_k\models \alpha(\bar g_k,\bar a_k)$ for all $k\geq L$. Define $$T_{\alpha,\bar a}=\left\{\bar \tau\in T_G^{|\bar g|}: \bar g_L\sqsubseteq \bar \tau \,\&  \bigwedge_{\bar g_L \sqsubseteq \bar \sigma \sqsubseteq \bar \tau} G_{\text{lv}(\bar \sigma)}\models \alpha\left(\bar \sigma, \bar a_{\text{lv}(\bar \sigma)}\right) \right\} $$ where $\text{lv}(\bar \sigma)$ is the natural number $k$ such that $\bar \sigma \in G_k\lom$. Note that $T_{\alpha,\bar a}$ is a finite branching tree with $\bar g_L$ as a root, and that it must be infinite (in particular, $\bar g$ gives a path). Additionally, since $T_G$ and $\bar a$ are computable in $I$ we have that $T_{\alpha,\bar a}$ is computable in $I$. Thus, since $I$ is a Scott ideal, there must be an $I$ computable path through $T_{\alpha,\bar a}$ and therefore a $\bar g'\in G_I\lom$ such that $G_k\models \alpha(\bar g'_k,\bar a_k)$ for all $k\geq L$. By Lemma \ref{no orb ind} we have that $G\models \alpha(\bar g', \bar a)$ and thus $G_I\models (\E \bar x)\alpha(\bar x,\bar a) $. \qed 
\end{proof}

Last, we show that when $G$ has orbit independence and $I$ is Scott ideal then $G_I\preceq G$. To prove this we will first establish a lemma which expands upon Corollary \ref{FV corollary}.

\begin{lemma}
\label{FV with free variables}
    Let $G$ be a profinite subgroup of $S_\infty$ with orbit independence. Given any $\bar a\in G\lom$ and first order sentence $\alpha$ in the language of groups, $G\models \alpha (\bar a)$ if and only if $G_k\models \alpha(\bar a_k)$ for all sufficiently large $k\in \N$. 
\end{lemma}

\begin{proof}
    Let $\bar a=\left(a^{(1)},...,a^{(n)}\right)$. Let $\LL$ be the language of groups and let $\LL^*=\LL\cup \{c^{(1)},...,c^{(n)}\}$ where each $c$ is a new constant symbol. Define $G^*$ to be $G$ but in the expanded language $\LL^*$ with $$G^*\models a^{(1)}=c^{(1)} \,\& \, \cdots \,\& \, a^{(n)}=c^{(n)}.$$ Similarly, for each $k\in \N$ define $H_k^*$ to be $H_k$ in the expanded language with $$H^*_k\models a^{(1)}\upharpoonright O_{G,k}=c^{(1)} \,\& \, \cdots \,\& \, a^{(n)}\upharpoonright O_{G,k}=c^{(n)}.$$ We have $G\models \alpha(\bar a)$ if and only if $G^*\models \alpha(\bar c)$. Note that $\alpha(\bar c)$ is a sentence in $\LL^*$, thus the same reasoning used in the proof of Corollary \ref{FV corollary} gives that $G^*\models \alpha(\bar c)$ holds if and only if $G^*_k\models \alpha(\bar c)$ for all but finitely many $k$ (where $G^*_k$ is defined analogously to $G_k$). Last, note that $G^*_k\models \alpha(\bar c)$ if and only if $G_k\models \alpha(\bar a_k)$. \qed
\end{proof}

\begin{theorem}
\label{gives elementary}
    Let $G$ be a profinite subgroup of $S_\infty$ with orbit independence. If $I$ is a Scott ideal with $\deg(T_G)\in I$, then $G_I \preceq G$. 
\end{theorem}

\begin{proof}
    We must show that $$G_I\models 
    \alpha(\bar a) \iff G\models 
    \alpha (\bar a)$$ for all 
    $\bar a\in G_I\lom$ and all first order formulas $\alpha$.
    We prove this by induction on the quantifier rank of $\alpha$, which we denote as $\qr(\alpha)$. Note, it is clear that the result holds when $\qr(\alpha)=0$. Suppose that it holds for all formulas of quantifier rank at most $n$. Let $\alpha(\bar x)=\E\bar y\beta(\bar y,\bar x)$ with $\qr(\beta)=n$ and fix $\bar a\in G_I$. If $G_I\models \alpha(\bar a)$, then there is some $\bar g\in G_I$ such that $G_I\models \beta(\bar g, \bar a)$. By inductive hypothesis, we have that $G\models \beta(\bar g, \bar a)$ and thus $G\models \alpha(\bar a)$. 

    For the other direction, suppose that $G\models \alpha(\bar a)$. We have that there exists a $\bar g\in G\lom $ such that $G\models \beta(\bar g, \bar a)$. Thus, by Lemma \ref{FV with free variables} there exists an $L\in \N$ such that $G_k\models \beta(\bar g_k,\bar a_k)$ for all $k\geq L$. Define $$T_{\beta,\bar a}=\left\{\bar \tau\in T_G^{|\bar g|}: \bar g_L\sqsubseteq \bar \tau \,\&  \bigwedge_{\bar g_L \sqsubseteq \bar \sigma \sqsubseteq \bar \tau} G_{\text{lv}(\bar \sigma)}\models \beta\left(\bar \sigma, \bar a_{\text{lv}(\bar \sigma)}\right) \right\} $$ where $\text{lv}(\bar \sigma)$ is the natural number $k$ such that $\bar \sigma \in G_k\lom$. Note that $T_{\beta,\bar a}$ is a finite branching tree with $\bar g_L$ as a root, and that it must be infinite (in particular, $\bar g$ gives a path). Additionally, recall that the branching function is computable in $\deg(T_G)$.  Since $T_G$ and $\bar a$ are computable in $I$ we have that $T_{\beta,\bar a}$ is computable in $I$. Thus, since $I$ is a Scott ideal, there must be an $I$ computable path through $T_{\beta,\bar a}$ and therefore a $\bar g'\in G_I\lom$ such that $G_k\models \beta(\bar g'_k,\bar a_k)$ for all $k\geq L$. By Lemma \ref{FV with free variables} we have that $G\models \beta(\bar g', \bar a)$ and thus $G_I\models \beta(\bar g', \bar a)$ by inductive hypothesis. Hence, $G_I\models \alpha(\bar a)$.  \qed
\end{proof}

\section{Further Questions}
\begin{question}
    How complicated can the $\E\forall$-theory of a profinite group be (relative to its tree presentation)? Is there any good upper bound like what Theorem \ref{sigma 2} gives us for the $\E$-theory?
\end{question}

\begin{question}
    Furthermore, how complicated can the entire first order theory be? Theorem \ref{delta 2} gives a good upper bound for the orbit independent case. Is there a good upper bound for the general case?
\end{question}

\begin{question}
    Given a Scott ideal $I$ and an arbitrary profinite $G$, is $G_I\preceq_{\E\forall} G$? If so, for how complicated a class of formulas is $G_I$  elementary in $G$? Might we always have $G_I\preceq G$?
\end{question}

\begin{question}
    Same as the above question, but with an arbitrary Turing ideal $I$ and an orbit independent $G$.
\end{question}

\begin{question}
    Are there any Turing ideals $I$ such that $G_I\preceq G$ for all profinite $G$?
\end{question}

\subsubsection{\ackname} The author wishes to acknowledge useful conversations with Meng-Che ``Turbo" Ho regarding the Feferman-Vaught Theorem.

\end{document}